\documentclass[11pt]{article}
\usepackage{amsmath,amsfonts,amssymb,amsthm}
\usepackage{eucal}

\def\Rp{{\mathbb R^+}}
\def\Zp{{\mathbb Z^+}}
\def\R{{\mathbb R}}
\def\P{{\mathbb P}}
\def\E{{\mathbb E}}
\def\I{{\mathbb I}}
\newtheorem{theorem}{Theorem}
\newtheorem{lemma}{Lemma}

\newcommand\mysection[1]{
              \refstepcounter{section}
              \section*{\centering\normalsize\bf\thesection.~#1}
                         }

\begin{document}
\section*{\centering\large\bf
           Tail behaviour of stationary distribution\\
           for Markov chains with asymptotically
           zero drift\footnote{Supported by DFG}}

\section*{\centering\normalsize\slshape\bfseries
Denis Denisov\footnote{School of Mathematics,
  Cardiff University, Senghennydd Road CARDIFF,
  Wales, UK.  CF24 4AG Cardiff. E-mail: DenisovD@cf.ac.uk},
Dmitry Korshunov\footnote{Sobolev Institute of Mathematics,
630090 Novosibirsk, Russia.
E-mail: korshunov@math.nsc.ru} and
Vitali Wachtel\footnote{Mathematical Institute,
  University of Munich,
  Theresienstrasse 39, D--80333 Munich, Germany
E-mail: wachtel@mathematik.uni-muenchen.de}}

\begin{abstract}
We consider a Markov chain on ${\mathbb R}^+$
with asymptotically zero drift and finite second
moments of jumps which is positive recurrent.
A power-like asymptotic behaviour of the invariant
tail distribution is proven; such a heavy-tailed invariant
measure happens even if the jumps of the chain are bounded.
Our analysis is based on test functions technique
and on construction of a harmonic function.

{\it Keywords}: Markov chain,
invariant distribution,
Lamperti problem,
asymptotically zero drift,
test (Lyapunov) function,
regularly varying tail behaviour,
convergence to $\Gamma$-distribution,
renewal function,
harmonic function

{\it AMS subject classification}:
Primary 60J05, 60F10; Secondary 60F15
\end{abstract}

%%%%%%%%%%%%%%%%%%%%%%%%%%%%%%%%%%%%%%%%%%%%%%%%%%%%%%%%%%%%%%%%%%%%%%%%%%%%%%%%%%%%%%%%%%%%%%%%%%%%%%%%%%
\mysection{Introduction, main results and discussion}
\label{sec:introduction}

Let $X=\{X_n, n\ge0\}$ be a time homogeneous Markov chain
taking values in $\Rp$. Denote by $\xi(x)$,
$x\in{\mathbb R}^+$, a random variable corresponding
to the jump of the chain at point $x$, that is,
a random variable with distribution
\begin{eqnarray*}
{\mathbb P}\{\xi(x)\in B\}
&=& {\mathbb P}\{X_{n+1}-X_n\in B\mid X_n=x\}\\
&=& {\mathbb P}_x\{X_{1}\in x+B\},
\quad B\in{\mathcal B}({\mathbb R});
\end{eqnarray*}
hereinafter the subscript $x$ denotes the initial position
of the Markov chain $X$, that is, $X_0=x$.

Denote the $k$th moment of the jump at point $x$
by $m_k(x):=\E\xi^k(x)$.
We say that a Markov chain has {\it asymptotically zero drift}
if $m_1(x)=\E\xi(x)\to 0$ as $x\to \infty$.
The study of processes with asymptotically zero drift was initiated by Lamperti
in a series of papers \cite{Lamp60, Lamp62,Lamp63}.

Processes with asymptotically zero drift naturally appear
in various stochastic models, here we mention only some
of them: branching processes, Klebaner \cite{Kleb84} and
K\"uster \cite{Kus85}; random billiards, Menshikov et al. \cite{MVW08};
random polymers, Alexander \cite{Alex11}, Alexander and Zygouras \cite{AZ09},
De Coninck et al. \cite{DDH08}.

We assume that the Markov chain $X_n$ possesses a stationary
(invariant) distribution and denote this distribution by $\pi$.
If we consider an irreducible aperiodic Markov chain on $\Zp$,
then existence of probabilistic invariant distribution
is equivalent to finiteness of $\E_0\tau_0$ where
$\tau_0:=\min\{n\ge 1:X_n=0\}$. For the state space $\Rp$,
we assume that $X_n$ is a positive Harris recurrent and
strongly aperiodic chain, see related definitions in \cite{MT}.
In particular, there exists a sufficiently large $x_0$
such that
\begin{eqnarray}\label{def.B}
\E_x\tau_B &<& \infty\ \mbox{ for all }x>x_0,
\end{eqnarray}
where $\tau_B:=\min\{n\geq1: X_n\in B\}$ and $B:=[0,x_0]$. We assume that the chain makes excertions
from any compact set, in the following sense.
We suppose that, for every fixed $x_1>x_0$,
there exists an $\varepsilon=\varepsilon(x_1)>0$ such that,
for every $x>x_0$,
\begin{eqnarray}\label{exersions}
\P_x\{X_{n(x)}>x_1,\tau_B>n(x)\} &\ge& \varepsilon\ \mbox{ for some }n(x).
\end{eqnarray}
We consider the case where $\pi$ has unbounded support, that is, $\pi(x,\infty)>0$
for every $x$. Our main goal is to describe
the asymptotic behaviour of its tail, $\pi(x,\infty)$,
for a class of Markov chains with asymptotically zero drift.

As it was shown in \cite[Theorem 1]{Kor11} any Markov chain with
asymptotically zero drift has heavy-tailed invariant
distribution provided
$$
\liminf_{x\to\infty}\ {\mathbb E}\{\xi^2(x);\xi(x)>0\} > 0;
$$
that is, all positive exponential moments of the invariant
distribution are infinite. The present paper is devoted
to the precise asymptotic behaviour of the invariant
tail distribution in the critical case where $m(x)$ behaves
like $-c/x$ for large $x$.
The existence of invariant distribution in critical case was
studied by Lamperti \cite{Lamp63}; this study is
based on considering the test function $V(x)=x^2$.
Then the drift of $V$ at point $x$ is equal to
${\mathbb E}\{V(X_{n+1})-V(X_n)\mid X_n=x\}=2xm_1(x)+m_2(x)$
and if $2xm(x)+b(x)<-\varepsilon$ for all sufficiently
large $x$, then the chain is positive recurrent and,
under mild technical conditions, it has unique invariant distribution
(see \cite[Chapter 11]{MT}).

There are two types of Markov chains for which the invariant
distribution is explicitly calculable. Both are related to
skip-free processes, either on lattice or on continious
state space $\Rp$.

The first case where the stationary distribution is explicitly
known is diffusion processes on $\Rp$ (slotted in time if we
need just a Markov chain). Let $m_1(x)$ and $m_2(x)$ be the
drift and diffusion coefficients at state $x$, respectively.
In the case of stable diffusion, the invariant density function
$p(x)$ solves the Kolmogorov forward equation
\begin{eqnarray*}
0 &=& -\frac{d}{dx}(m_1(x)p(x))+\frac12\frac{d^2}{dx^2}(m_2(x)p(x)),
\end{eqnarray*}
which has the following solution:
\begin{eqnarray}\label{ex:1}
p(x) &=& \frac{2c}{m_2(x)}e^{\int_0^x\frac{2m_1(y)}{m_2(y)}dy},
\qquad c>0.
\end{eqnarray}

The second case is the Markov chain on $\Zp$ with
$\xi(x)$ taking values $-1$, $1$ and $0$ only,
with probabilities $p_-(x)$, $p_+(x)$ and $1-p_-(x)-p_+(x)$
respectively, $p_-(0)=0$. Then the stationary probabilities
$\pi(x)$, $x\in{\mathbb Z^+}$, satisfy the equations
$$
\pi(x)=\pi(x-1)p_+(x-1)+\pi(x)(1-p_+(x)-p_-(x))+\pi(x+1)p_-(x+1),
$$
which have the following solution:
\begin{eqnarray}\label{ex:2}
\pi(x) &=& \pi(0)\prod_{k=1}^x\frac{p_+(k-1)}{p_-(k)}
= \pi(0)e^{\sum_{k=1}^x(\log p_+(k-1)-\log p_-(k))},
\end{eqnarray}
where under some regularity conditions the sum may be
approximated by the integral like in the diffusion case.

In paper \cite{MP95}, Menshikov and Popov investigated
behaviour of the invariant distribution
$\{\pi(x),x\in{\mathbb Z}^+\}$ for countable Markov chains
with asymptotically zero drift and with bounded jumps
(see also Aspandiiarov and Iasnogorodski \cite{AI99}).
Some rough theorems for the local probabilities $\pi(x)$
were proved; if
\begin{equation}
\label{moments}
m_1(x)\sim-\frac{\mu}{x},\quad m_2(x)\sim b
\quad\text{and}\quad 2\mu>b
\end{equation}
then for every $\varepsilon>0$ there exist
constants $c_-=c_-(\varepsilon)>0$ and
$c_+=c_+(\varepsilon)<\infty$ such that
$$
c_-x^{-2\mu/b-\varepsilon}\le\pi(x)
\le c_+x^{-2\mu/b+\varepsilon}.
$$

The paper \cite{Kor11} is devoted to the existence
and non-existence of moments of invariant distribution.
In particular, there was proven that if (\ref{moments}) holds and
the families of random variables
$\{(\xi^+(x))^{2+\gamma},x\ge0\}$ for some $\gamma>0$ and
$\{(\xi^-(x))^2,x\ge0\}$ are uniformly integrable
then the moment of order $\gamma$ of the
invariant distribution $\pi$ is finite if
$\gamma<2\mu/b-1$, and infinite if $\pi$ has unbounded
support and $\gamma>2\mu/b-1$. This result implies that
for every $\varepsilon>0$ there exists $c(\varepsilon)$
such that
\begin{equation}
\label{up.bound}
\pi(x,\infty)\leq c(\varepsilon) x^{-2\mu/b+1+\varepsilon}.
\end{equation}

To the best of our knowledge there are no results in the
literature on the exact asymptotic behaviour for the measure $\pi$.

\begin{theorem}\label{thm:tail}
Suppose that \eqref{moments} holds.
Suppose also that there exists a differentiable function
$r(x)>0$ such that $r'(x)\sim -\frac{2\mu}{bx^2}$ and
\begin{eqnarray}\label{r-cond.2}
\frac{2m_1(x)}{m_2(x)} &=& -r(x)+O(1/x^{2+\delta})
\end{eqnarray}
for some $\delta>0$. Suppose also that
\begin{eqnarray}\label{moment_cond1}
\sup_x\E|\xi(x)|^{3+\delta} &<& \infty,
\end{eqnarray}
\begin{equation}\label{3.moment}
\E\xi^3(x)\to m_3\in(-\infty,\infty)
\end{equation}
and, for some $A<\infty$,
\begin{eqnarray}\label{cond.xi.U+2}
\E\{\xi^{2\mu/b+3+\delta}(x);\xi(x)>Ax\} &=& O(x^{2\mu/b}).
\end{eqnarray}
Then there exist a constant $c>0$ such that
$$
\pi(x,\infty)\sim c x e^{-\int_0^x r(y)dy}= c x^{-2\mu/b+1}\ell(x)
\quad\mbox{ as }x\to\infty,
$$
where
$\ell(x):=x^{2\mu/b}/e^{\int_0^x r(y)dy}$ is a slowly varying function.
\end{theorem}
It is clear that the convergence of third moments is a technical condition
because the asymptotic behaviour of the stationary measure depends on $m_1(x)$
and $m_2(x)$ only and does not depend on $m_3$. Also as follows from the
moments existence results \cite{Kor11}, it is likely that the statement of
Theorem~\ref{thm:tail} should follow under less restrictive condition than
\eqref{cond.xi.U+2}, with $2\mu/b+1+\delta$ moments instead.
Unfortunately, we cannot just remove restriction \eqref{3.moment}
from the theorem, but we can weaken it by introducing some structural restrictions,
the main of which is the left-continuity of $X_n$.

\begin{theorem}\label{thm:tail2}
Suppose that all conditions of Theorem {\rm\ref{thm:tail}} hold
except probably the condition \eqref{3.moment}.
If, in addition, $X_n$ lives on ${\mathbb Z}^+$ and
$\xi(x)\ge -1$, then the statement of Theorem {\rm\ref{thm:tail}}
remains valid.
\end{theorem}

To prove Theorems~\ref{thm:tail} and \ref{thm:tail2}
we change the probability measure
in such a way that the resulting object will be a transient
Markov chain with asymptotically zero drift.
We apply the following change of measure:
$$
\widehat P(x,dy):=\frac{V(y)\P_x\{X_1\in dy,\tau_B>1\}}{V(x)},
$$
where $V$ is a harmonic function for the substochastic kernel
$\P_x\{X_1\in dy,\tau_B>1\}$.
In this way we need to produce a suitable harmonic function $V$.
Since the harmonic function for the corresponding Bessel-type process
conditioned to stay positive is known, we adapt the method
proposed in \cite{DW12} where random walks conditioned to stay
in a cone were considered. (This method allows one to construct
harmonic functions for random walks from harmonic functions
for corresponding limiting diffusions.) Again,  the only processes,
where harmonic functions were known, are diffusions and Markov chains
with jumps $\pm 1$. The latter case has been considered by
Alexander \cite{Alex11}.

Investigation of large deviation probabilities for one-dimensional
Markov chains with ultimately negative drift heavily depends on
whether this chain is similar to the process of summation with more
or less homogeneous drift (and in this case we may speak about the
process with continuous statistics) or this Markov chain is close to
a random walk on $\mathbb{R}^+$ with delay at the origin where the
mean drift change its sign near the origin (in this case we have the
chain with discontinuous statistics). The only Markov chain which can
be somehow reduced to the sums is the chain $W_{n}=(W_{n-1}+\xi_n)^+$
with independent identically distributed $\xi$'s which equals in
distribution to $\max_{k\leq n}\sum_{j=1}^k\xi_j$. For these two
classes of Markov chains (with continuous and discontinuous statistics)
the methods for investigation of large deviations are essentially
different. Say, in Cramer case where some exponential moment of jumps
is bounded, an appropriate exponential change of measure leads preserves
the measures to be probabilistic. If we apply exponential change of
measures to a chain with discontinuous statistics may lead to
non-stochastic kernel. Such approach was utilised in \cite{Kor04} and
there appears a necessity for proving limit theorems for non-stochastic
transition kernels.

In the setting of the present paper one could think of applying
of a change of measure method with power-like weight function.
Then the probability measure changes in such a way that the resulting
object will be similar to a transient Markov chain with
asymptotically zero drift. One may look at the following
two approaches:
\begin{itemize}
\item[(a)]
$\displaystyle Q^{(1)}(x,B)
:=\frac{\E\{X_1^\rho{\rm 1}\{X_1\in B\}|X_0=x\}}{x^{\rho}}$,
$\rho=2\mu/b+1$;
\item[(b)]
$\displaystyle Q^{(2)}(x,B)
:=\frac{\E\{V(X_1){\rm 1}\{X_1\in B\}|X_0=x\}}{V(x)}$.
\end{itemize}
In the first case we would have measures which are not necessarily probabilistic, i.e.,
$Q^{(1)}(x,\mathbb{R}^+)$ can be smaller or greater than $1$; this case is similar to that
considered in \cite{Kor04} for the case of the exponential change of measure.

With
$\rho=2\mu/b+1$ one can show that the Markov evolution of masses is asymptotically equivalent to
a transient Markov chain with asymptotically zero drift. And our hope is that one can adopt results,
which will be proved in the present project, to Markov evolutions of masses. If this is the case, then
we can translate the results for Markov evolutions of masses into results for positive recurrent
Markov chains by applying the inverse change of measure.

As it was mentioned above, in this paper we develop the second
possibility for the change of measure, where we get stochastic
transition kernel corresponding to a transient Markov chain.
Then the main difficulties are related to the fact that the
harmonic function $V$ is given implicitely. In particular, we even
need to check that $V$ is regular varying function with index $\rho$.

Having this observation in mind we face to necessity of
obtaining limiting results for transient Markov chains.
In Section \ref{sec:transience} we give rather general
close to necessary conditions for transience while in Section
\ref{sec:escape} we make some quantitative analysis
of how fast a transient chain escapes to the infinity.
Section \ref{sec:gamma} is devoted to convergence to
for the $\Gamma$-distribution under optimal assumptions:
null-recurrence or transience of the process
and minimal integrability restrictions.
Section \ref{sec:renewal} contains integral renewal theorem for
transient Markov chain with drift $c/x$, $c>0$.
In Section \ref{sec:harmony} a general results on harmonic
functions are discussed.
In order to obtain results for the original positive recurrent
Markov chain one needs to apply again the inverse change of measure.
This is done in Section \ref{sec:proof}.

\mysection{Conditions for transience revised}
\label{sec:transience}

In general, if, for some $x_0$ and $\varepsilon>0$,
\begin{eqnarray}\label{rec.1}
\frac{2m_1(x)}{m_2(x)}\ge 1+\varepsilon
\ \mbox{ for all }x\ge x_0,
\end{eqnarray}
then the drift to the right dominates the diffusion
and the corresponding Markov chain $X_n$ is typically transient.
As an example concluding this section shows,
for transience, the Markov chain should satisfy some additional
conditions on jumps. In the literature, the transience
in Lamperti problem was studied by Lamperti \cite{Lamp60},
Kersting \cite{Ker86} and Menshikov et al. \cite{AMI}
under different conditions, say for the case of bounded jumps
or of moments of order $2+\delta$ bounded.
Our goal here is to clarify what condition in addition to
\eqref{rec.1} is responsible for transience. Surprisingly,
such a condition is rather weak and is presented in \eqref{rec.3}.

\begin{theorem}\label{thm:transience}
Assume the condition \eqref{rec.1} holds. In addition, let
\begin{eqnarray}\label{rec.2}
\P\bigl\{\limsup_{n\to\infty}X_n=\infty\bigr\} &=& 1
\end{eqnarray}
and, for some $\gamma$, $0<\gamma<1-1/\sqrt{1+\varepsilon}$,
\begin{eqnarray}\label{rec.3}
\P\{\xi(x)\le-\gamma x\} &=& o\Bigl(\frac{m_2(x)}{x}p(x)\Bigr)
\ \mbox{ as }x\to\infty,
\end{eqnarray}
where a non-increasing function $p(x)$ is integrable.
Then $X_n\to\infty$ as $n\to\infty$ with probability 1,
so that $X_n$ is transient.
\end{theorem}

The condition \eqref{rec.2} (which was first proposed
in this framework by Lamperti \cite{Lamp60}) may be equivalently restated
as follows: for any $N$ the exit time from the set $[0,N]$
is finite with probability 1. In this way it is clear that,
for a countable Markov chain, the irreducibility implies
\eqref{rec.2}. For a Markov chain on general state space,
the related topic is $\psi$-irreducibility, see
\cite[Secs 4 and 8]{MT}.

\begin{proof}[Proof of Theorem \ref{thm:transience}]
is based on the standard approach of construction
of a nonnegative bounded test function $V_*(x)\downarrow 0$
such that $V_*(X_n)$ is a supermartingale with further
application of Doob's convergence theorem for supermartingales.

Since $p(x)$ is non-increasing and integrable, by \cite{D2006},
there exists a continuous non-increasing integrable regularly
varying at infinity with index $-1$ function $V_1(x)$
such that $p(x)\le V_1(x)$. Take
$$
V(x):=\int_x^\infty V_2(y)dy,
\quad\mbox{ where }\quad
V_2(x):=\int_x^\infty\frac{V_1(y)}{y}dy.
$$
By Theorem 1(a) from \cite[Ch  VIII, Sec 9]{Feller}
we know that $V_2$ is regularly varying with index $-1$
and $V_2(x)\sim V_1(x)$ as $x\to\infty$.
Since $V_1$ is integrable, the nonnegative non-increasing
function $V(x)$ is bounded, $V(0)<\infty$, and,
by the same reference, $V(x)$ is slowly varying.

Let us prove that the mean drift of $V(x)$ is negative
for all sufficiently large $x$. We have
\begin{eqnarray*}
\lefteqn{\E V(x+\xi(x))-V(x)}\\
&\le& V(0)\P\{\xi(x)\le -\gamma x\}
+\E\{V(x+\xi(x))-V(x);\xi(x)>-\gamma x\}\\
&=& V(0)\P\{\xi(x)\le -\gamma x\}
+V'(x)\E\{\xi(x);\xi(x)>-\gamma x\}\\
&&\hspace{50mm} +\frac12\E\{\xi^2(x) V''(x+\theta\xi(x));
\xi(x)>-\gamma x\},
\end{eqnarray*}
where $0\le\theta=\theta(x,\xi(x))\le 1$,
by Taylor's formula with the remainder in the Lagrange form.
By the construction, $V'(x)=-V_2(x)<0$,
$\E\{\xi(x);\xi(x)>-\gamma x\}\ge m_1(x)>0$ for $x\ge x_0$,
and $V''(x)=V_1(x)/x$ is non-increasing. Hence,
\begin{eqnarray*}
\lefteqn{\E V(x+\xi(x))-V(x)}\\
&\le& V(0)\P\{\xi(x)\le -\gamma x\}
-V_2(x)m_1(x)+\frac{V''((1-\gamma)x)}{2}m_2(x)\\
&=& o\Bigl(\frac{m_2(x)V_1(x)}{x}\Bigr)
-\frac{m_2(x)V_1(x)}{2x}\Bigl(\frac{2xm_1(x)}{m_2(x)}
\frac{V_2(x)}{V_1(x)}
-\frac{x}{V_1(x)}\frac{V_1((1-\gamma)x)}{(1-\gamma)x}\Bigr),
\end{eqnarray*}
by the condition \eqref{rec.3} and the inequality $p(x)\le V_1(x)$.
Applying now the condition \eqref{rec.1} together with
the equivalences $V_2(x)\sim V_1(x)$ and
$V_1((1-\gamma)x)\sim V_1(x)/(1-\gamma)$ we deduce that there
exists a sufficiently large $x_*$ such that, for all $x\ge x_*$,
\begin{eqnarray*}
\E V(x+\xi(x))-V(x)
&\le& -\frac{m_2(x)V_1(x)}{2x}\varepsilon_*,
\end{eqnarray*}
where $\varepsilon_*:=(1+\varepsilon-(1-\gamma)^{-2})/2>0$.
Now take $V_*(x):=\min(V(x),V(x_*))$. Then
\begin{eqnarray*}
\E V_*(x+\xi(x))-V_*(x) &\le& \E V(x+\xi(x))-V(x)<0
\end{eqnarray*}
for every $x\ge x_*$ and
\begin{eqnarray*}
\E V_*(x+\xi(x))-V_*(x) &=& \E\{V(x+\xi(x))-V(x_*);x+\xi(x)\ge x_*\}
\le 0
\end{eqnarray*}
for every $x<x_*$. Therefore, $V_*(X_n)$ constitutes a
nonnegative bounded supermartingale and, by Doob's
convergence theorem, $V_*(X_n)$ has an a.s. limit as $n\to\infty$.
Due to the condition \eqref{rec.2}, this limit equals
$V_*(\infty)=0$ and the proof is complete.
\end{proof}

Roughly speaking, the condition \eqref{rec.3} guarantees that
large negative jumps don't make any valuable
contribution to the evolution of the chain compared to
the contribution of the first and second moments of jumps.
Let us demonstrate by example that the condition \eqref{rec.3}
is very essential and in a sense almost necessary.

Consider a Markov chain $X_n$ on $\Rp$ satisfying the following
conditions: for some function $f(x)\ge 0$, $m_1(x) \le f(x)$ and
\begin{eqnarray}\label{origin}
\P\{\xi(x)=-x\} &=& f(x)p(x)
\end{eqnarray}
for all sufficiently large $x$, where $p(x)$ is a non-increasing
function satisfying $p(x)=O(1/x)$ and
$$
V(x):=\int_0^x p(y)dy\to\infty\ \mbox{ as }x\to\infty.
$$

In this example the high probability of the large negative jump
$-x$ leads to recurrence of the chain (note that if
$f(x)=m_2(x)/x$ then the condition \eqref{rec.3} fails to hold).

Indeed, decompose the mean drift of the increasing concave
test function $V$ at state $x$ separating the jump to the origin:
\begin{eqnarray}\label{deco.1}
\E V(x+\xi(x))-V(x) &=& -V(x)\P\{\xi(x)=-x\}\nonumber\\
&& +\E\{V(x+\xi(x))-V(x);\xi(x)>-x\}.
\end{eqnarray}
Since $V(x)$ is concave and $V'(x)=p(x)$, by Jensen's inequality,
\begin{eqnarray*}
\E\{V(x+\xi(x))-V(x);\xi(x)>- x\}
&\le& p(x)\E\{\xi(x);\xi(x)>-x\}\\
&=& p(x)(m_1(x)+x\P\{\xi(x)=-x\})\\
&=& O(p(x)f(x))\ \mbox{ as }x\to\infty,
\end{eqnarray*}
because $xp(x)$ is bounded. Substituting this together with
\eqref{origin} into \eqref{deco.1}, we obtain the following
upper bound for the drift:
\begin{align*}
\E V(x+\xi(x))-V(x) \le p(x)f(x)(-V(x)+O(1)).
\end{align*}
Since $V(x)\to\infty$ as $x\to\infty$, the drift becomes
asymptotically negative and the chain $X_n$ is recurrent,
see e.g. \cite[Theorem 8.4.3]{MT}.

\mysection{Quantitative analysis of escaping to infinity
for transient chain}
\label{sec:escape}

First we give an upper bound for the return probability
for transient Markov chain.

\begin{lemma}\label{est.for.return}
Assume the condition \eqref{rec.1} holds and,
for some $\delta$, $\gamma>0$ satisfying
$(1+\delta)/(1-\gamma)^{2+\delta}<1+\varepsilon$,
\begin{eqnarray}\label{rec.4}
\P\{\xi(x)\le-\gamma x\} &=& o\Bigl(\frac{m_2(x)}{x^{2+\delta}}\Bigr)
\ \mbox{ as }x\to\infty.
\end{eqnarray}
Then there exist $x_0$ such that
\begin{eqnarray*}
\P\{X_n\le x\mbox{ for some }n\ge1\mid X_0=y\} &\le& (x/y)^\delta
\quad \mbox{ for all }y>x>x_0.
\end{eqnarray*}
\end{lemma}

\begin{proof}
Fix $y>0$. Consider a test function $V(x):=\min(x^{-\delta},1)$.
The mean drift of $V(x)$ is negative for all sufficiently large $x$.
Indeed,
\begin{eqnarray*}
\E V(x+\xi(x))-V(x)
&\le& \P\{\xi(x)\le -\gamma x\}
+\E\{V(x+\xi(x))-V(x);\xi(x)>-\gamma x\}\\
&=& \P\{\xi(x)\le -\gamma x\}
-\frac{\delta}{x^{1+\delta}}\E\{\xi(x);\xi(x)>-\gamma x\}\\
&&\hspace{20mm} +\frac{\delta(1+\delta)}{2}
\E\Bigl\{\frac{\xi^2(x)}{(x+\theta\xi(x))^{2+\delta}};
\xi(x)>-\gamma x\Bigr\},
\end{eqnarray*}
where $0\le\theta=\theta(x,\xi(x))\le 1$, by Taylor's formula.
Therefore,
\begin{eqnarray*}
\E V(x+\xi(x))-V(x)
&\le& \P\{\xi(x)\le -\gamma x\}
-\frac{\delta}{x^{1+\delta}}m_1(x)
+\frac{\delta(1+\delta)}{2((1-\gamma)x)^{2+\delta}}m_2(x)\\
&=& o\Bigl(\frac{m_2(x)}{x^{2+\delta}}\Bigr)
-\frac{\delta m_2(x)}{2x^{2+\delta}}\Bigl(\frac{2xm_1(x)}{m_2(x)}
-\frac{1+\delta}{(1-\gamma)^{2+\delta}}\Bigr),
\end{eqnarray*}
by the condition \eqref{rec.4}.
Then the condition \eqref{rec.1} implies that there
exists sufficiently large $x_*$ such that, for all $x\ge x_*$,
\begin{eqnarray*}
\E V(x+\xi(x))-V(x)
&\le& -\frac{\gamma m_2(x)}{2x^{2+\delta}}\varepsilon_*,
\end{eqnarray*}
where $\varepsilon_*:=(1+\varepsilon-(1+\delta)/(1-\gamma)^{2+\delta})/2>0$.
Now take $V_*(x):=\min(V(x),V(x_*))$ so that $V_*(X_n)$
is nonnegative bounded supermartingale.
Hence we may apply Doob's inequality for nonnegative
supermartingale and deduce that, for every $y>x\ge x_*$
(so that $V_*(y)<V_*(x)$),
\begin{eqnarray*}
\P\{\sup_{n\ge 1} V_*(X_n)>V_*(x)\mid V_*(X_0)=V_*(y)\}
&\le& \frac{\E V_*(X_0)}{V_*(x)}=\Bigl(\frac{x}{y}\Bigr)^\delta,
\end{eqnarray*}
which is equivalent to the lemma conclusion.
\end{proof}

In the next lemma we estimate from above the mean value
$\E_y T(x)$ of the first up-crossing time
$$
T(x):=\min\{n\ge 1: X_n>x\}.
$$

\begin{lemma}\label{lem.uniform}
Assume that, for some $x_0\ge 0$, $\varepsilon_0\ge0$,
and $\varepsilon>0$,
\begin{eqnarray}\label{T.above.cond.1}
2xm_1(x)+m_2(x) &\ge&
\left\{
\begin{array}{cl}
\varepsilon,&\mbox{ if }x>x_0,\\
-\varepsilon_0,&\mbox{ if }x\le x_0.
\end{array}
\right.
\end{eqnarray}
Then, for every $x>y$,
\begin{eqnarray*}
\E _y T(x) &\le&
\frac{x^2-y^2+c(x)+(\varepsilon+\varepsilon_0)H_y(x_0)}{\varepsilon},
\end{eqnarray*}
where
\begin{eqnarray}\label{T.above.cond.2}
c(x) &:=& \sup_{z\le x}(2zm_1(z)+m_2(z))
\end{eqnarray}
and
\begin{eqnarray*}
H_y(x_0):=\sum_{n=0}^\infty \P_y\{X_n\le x_0\}.
\end{eqnarray*}
\end{lemma}

\begin{proof}
Consider the following random sequence:
$$
Y_n:=X_n^2+(\varepsilon_0+\varepsilon)
\sum_{k=0}^{n-1} \I\{X_k\le x_0\}.
$$
First, $Y_n$ is a submartingale with respect to the filtration
${\mathcal F}_n:=\sigma(X_k,k\le n)$. Indeed,
$$
Y_{n+1}-Y_n=X_{n+1}^2-X_n^2
+(\varepsilon_0+\varepsilon)\I\{X_n\le x_0\},
$$
so that
\begin{eqnarray}\label{Y.ge.eps}
\E\{Y_{n+1}-Y_n\mid{\mathcal F}_n\} &=& 2X_nm_1(X_n)+m_2(X_n)
+(\varepsilon_0+\varepsilon)\I\{X_n\le x_0\}\nonumber\\
&\ge& \varepsilon>0,
\end{eqnarray}
by the condition \eqref{T.above.cond.1}. Thus, for any $x>y$,
\begin{eqnarray}\label{EYT.below}
\E_y Y_{T(x)} &\ge& y^2+\varepsilon \E_y T(x),
\end{eqnarray}
due to the adapted version of the proof of Dynkin's formula
(see, e.g. \cite[Theorem 11.3.1]{MT}):
\begin{eqnarray*}
\E_y Y_{T(x)} &=& \E_y Y_0+\E_y\sum_{n=1}^\infty
\I\{n\le T(x)\}(Y_n-Y_{n-1})\\
&=& y^2+\E_y\sum_{n=1}^\infty
\E\{\I\{n\le T(x)\}(Y_n-Y_{n-1})\mid{\mathcal F}_{n-1}\}\\
&=& y^2+\E_y\sum_{n=1}^\infty \I\{T(x)\ge n\}
\E\{Y_n-Y_{n-1}\mid{\mathcal F}_{n-1}\},
\end{eqnarray*}
because $\I\{n\le T(x)\}\in{\mathcal F}_{n-1}$.
Hence, it follows from \eqref{Y.ge.eps} that
\begin{eqnarray*}
\E_y Y_{T(x)} &\ge& y^2
+\varepsilon\E_y\sum_{n=1}^\infty \I\{T(x)\ge n\}\\
&=& y^2+\varepsilon\sum_{n=1}^\infty \P_y\{T(x)\ge n\},
\end{eqnarray*}
and the inequality \eqref{EYT.below} follows.

On the other hand,
\begin{eqnarray}\label{EYT.upper.1}
\E_y Y_{T(x)} &=& \E_y X^2_{T(x)}
+(\varepsilon_0+\varepsilon)\E_y
\sum_{k=0}^{T(x)-1} \I\{X_k\le x_0\}.
\end{eqnarray}
Further,
\begin{eqnarray*}
\E\{X^2_{T(x)}\mid X_{T(x)-1}\} &=&
\E\{(X_{T(x)-1}+\xi(X_{T(x)-1}))^2 \mid X_{T(x)-1}\}\\
&=& X^2_{T(x)-1}+\E\{2X_{T(x)-1}m_1(X_{T(x)-1})
+m_2(X_{T(x)-1}) \mid X_{T(x)-1}\}\\
&\le& x^2+c(x),
\end{eqnarray*}
by the definition \eqref{T.above.cond.2} of $c(x)$.
Substituting this into \eqref{EYT.upper.1} we deduce
$$
\E_y Y_{T(x)}\le x^2+c(x)+(\varepsilon_0+\varepsilon)H_y(x_0),
$$
which together with \eqref{EYT.below} yields
the lemma conclusion. The proof is complete.
\end{proof}

\begin{lemma}\label{lem.exp}
Let the conditions of Lemma \ref{lem.uniform} hold and
$c(x)=O(x^2)$ in the condition \eqref{T.above.cond.2} and
\begin{eqnarray}\label{T.above.cond.3}
\sup_{y\le x_0}H_y(x_0)=\sup_{y\le x_0}
\sum_{n=0}^\infty \P_y\{X_n\le x_0\} < \infty.
\end{eqnarray}
Then there exist $c>0$ and $t_0$ such that,
for any $t>0$ and $y<x$,
$$
\P_y\{T(x)>tx^2\}\le e^{-c(t-t_0)}.
$$
\end{lemma}

\begin{proof}
Considering the first visit to the interval $[0,x_0]$
we deduce from the condition \eqref{T.above.cond.3} that
$$
\sup_{y\ge 0}\sum_{n=0}^\infty \P_y\{X_n\le x_0\} < \infty.
$$
Thus, by Lemma \ref{lem.uniform}, there exists $c_1<\infty$
such that, for all $x$,
\begin{eqnarray}\label{T.x2.uni}
\sup_y\E_y T(x) &\le& c_1(x^2+1).
\end{eqnarray}

Next, by the Markov property, for every $t$ and $s>0$,
\begin{eqnarray*}
\P_y\{T(x)>t+s\} &=&
\int_0^x \P_y\{T(x)>t, X_t\in dz\}\P_z\{T(x)>s\}\\
&\le& \P_y\{T(x)>t\}\sup_{z\le x}\P_z\{T(x)>s\}.
\end{eqnarray*}
Therefore, the monotone function
$q(t):=\sup_{y\le x}\P_y\{T(x)>tx^2\}$
satisfies the relation $q(t+s)\le q(t)q(s)$.
Then the increasing function $r(t):=\log(1/q(t))$ is convex
and $r(0)=0$.
By the bound \eqref{T.x2.uni} and Chebyshev's inequality,
there exists $t_0$ such that $q(t_0)<1$ so that
$q(t_0)=e^{-c}$ with $c>0$, and $r(t_0)=c>0$.
Then, by $r(0)=0$ and convexity of $r$, $r(t)\ge c(t-t_0)$
which implies $q(t)\le e^{-c(t-t_0)}$.
The proof is complete.
\end{proof}

\mysection{Convergence to $\Gamma$-distribution for
transient and null-recurrent chains}
\label{sec:gamma}

In this section we are interested in the asymptotic
growth rate of a Markov chain $X_n$
that goes to infinity in distribution as $n\to\infty$.
It happens if this chain is either transient or null recurrent.
First time a limit theorem for Markov chain with asymptotically
zero drift was produced by Lamperti in \cite{Lamp62} where the convergence
to $\Gamma$-distribution was proven for the case of jumps
with all moments finite. The proof is based on the method
of moments. The results from \cite{Lamp62}
have been generalised by Klebaner \cite{Kleb89} and later by
Kersting \cite{Ker92b}. The author of the latter paper works
under the assumption that the moments of order $2+\delta$ are bounded.
But the convergence is proven on the event $\{X_n\to\infty\}$ only.
This restriction is not necessary, since Lamperti's result
allows $X_n$ to be null-recurrent, and for null-recurrent
processes we have $\P\{X_n\to\infty\}=0$.

\begin{theorem}\label{thm:gamma}
Assume that, for some $b>0$ and $\mu>-b/2$,
\begin{eqnarray}\label{1.2}
{\mathbb E}\xi(x)\sim \mu/x\ \mbox{ and }\
{\mathbb E}\xi^2(x)\to b
\quad\mbox{ as }x\to\infty
\end{eqnarray}
and that the family $\{\xi^2(x), x\ge 0\}$
possesses an integrable majorant $\Xi$, that is,
${\mathbb E}\Xi<\infty$ and
\begin{eqnarray}\label{mom.cond.ui}
\xi^2(x) &\le_{st}& \Xi
\quad\mbox{ for all }x.
\end{eqnarray}
If $X_n\to\infty$ in probability as $n\to\infty$, then
$X_n^2/n$ converges weakly to the $\Gamma$-distribution
with mean $2\mu+b$ and variance $(2\mu+b)2b$.
\end{theorem}

\begin{proof}
For any $n\in{\mathbb N}$,
consider a new Markov chain $Y_k(n)$, $k=0$, $1$, $2$, \ldots,
with transition probabilities depending on the parameter $n$,
whose jumps $\eta(n,x)$ are just truncations
of the original jumps $\xi(x)$ at level $x\vee \sqrt n$
depending on both point $x$ and time $n$, that is,
$$
\eta(n,x)=\min\{\xi(x),x\vee \sqrt n\}.
$$
Given $Y_0(n)=X_0$, the probability of discrepancy
between the trajectories of $Y_k(n)$ and $X_k$
until time $n$ is at most
\begin{eqnarray}\label{BK}
{\mathbb P}\{Y_k(n)\neq X_k\mbox{ for some }k\le n\}
&\le& \sum_{k=0}^{n-1}{\mathbb P}\{X_{k+1}-X_k\ge\sqrt n\}\nonumber\\
&\le& n{\mathbb P}\{\Xi\ge n\}\nonumber\\
&\le& {\mathbb E}\{\Xi;\Xi\ge n\} \to 0
\ \mbox{ as }n\to\infty.
\end{eqnarray}

Since $X_n$ converges in probability to infinity,
\eqref{BK} implies that, for every $c$,
\begin{eqnarray}\label{Y.to.infty}
\inf_{n>n_0,k\in[n_0,n]}
{\mathbb P}\{Y_k(n)>c\} &\to& 1
\quad\mbox{ as }n_0\to\infty.
\end{eqnarray}

By the choice of the truncation level,
$$
\xi(x) \ge \eta(n,x) \ge \xi(x)-\xi(x){\mathbb I}\{\xi(x)>x\}.
$$
Therefore, by the condition \eqref{mom.cond.ui},
\begin{eqnarray}\label{1.Y}
{\mathbb E}\eta(n,x)
&=& {\mathbb E}\xi(x)+o(1/x)
\quad\mbox{ as }x\to\infty \mbox{ uniformly in }n
\end{eqnarray}
and
\begin{eqnarray}\label{2.Y}
{\mathbb E}\eta^2(n,x)
&=& {\mathbb E}\xi^2(x)+o(1)
\quad\mbox{ as }x\to\infty \mbox{ uniformly in }n.
\end{eqnarray}
In addition, the inequality $\eta(n,x)\le x\vee\sqrt n$
and the condition \eqref{mom.cond.ui} imply that,
for every $j\ge 3$,
\begin{eqnarray}\label{k.Y}
{\mathbb E}\eta^j(n,x)
&=& o(x^{j-2}+n^{(j-2)/2})
\quad\mbox{ as }x\to\infty \mbox{ uniformly in }n.
\end{eqnarray}

Compute the mean of the increment of $Y_k^j(n)$.
For $j=2$ we have
\begin{eqnarray*}
{\mathbb E}\{Y^2_{k+1}(n)-Y^2_k(n)|Y_k(n)=x\}
&=& {\mathbb E}(2x\eta(n,x)+\eta^2(n,x))\\
&=& 2\mu+b+o(1)
\quad\mbox{ as }x\to\infty \mbox{ uniformly in }n,
\end{eqnarray*}
by \eqref{1.Y} and \eqref{2.Y}.
Applying now \eqref{Y.to.infty} we get
\begin{eqnarray*}
{\mathbb E}(Y^2_{k+1}(n)-Y^2_k(n)) &\to& 2\mu+b
\quad\mbox{ as }k,n\to\infty,k\le n.
\end{eqnarray*}
Hence,
\begin{eqnarray}\label{asy.2}
{\mathbb E}Y^2_n(n) &\sim& (2\mu+b)n
\quad\mbox{ as }n\to\infty.
\end{eqnarray}
Let now $j=2i$, $i\ge2$. We have
\begin{eqnarray}\label{incr.n.2i.x}
\lefteqn{{\mathbb E}\{Y^{2i}_{k+1}(n)-Y^{2i}_k(n)|Y_k(n)=x\}}\nonumber\\
&=& {\mathbb E}\Biggl(2ix^{2i-1}\eta(n,x)
+i(2i-1)x^{2i-2}\eta^2(n,x)
+\sum_{l=3}^{2i}x^{2i-l}\eta^l(n,x)\binom{2i}{l}\Biggr)\nonumber\\
&=& i[2\mu+(2i-1)b+o(1)]x^{2i-2}
+\sum_{l=3}^{2i}x^{2i-l}{\mathbb E}\eta^l(n,x)\binom{2i}{l}
\end{eqnarray}
as $x\to\infty$ uniformly in $n$,
by \eqref{1.Y} and \eqref{2.Y}. Owing to \eqref{k.Y},
\begin{eqnarray*}
\sum_{l=3}^{2i}x^{2i-l}{\mathbb E}\eta^l(n,x)\binom{2i}{l}
&=& \sum_{l=3}^{2i}x^{2i-l}o(x^{l-2}+n^{(l-2)/2})\\
&=& o(x^{2i-2})+\sum_{l=3}^{2i}x^{2i-l}o(n^{(l-2)/2})
\end{eqnarray*}
as $n\to\infty$ uniformly in $x$.
Thus,
\begin{eqnarray*}
\sum_{l=3}^{2i}x^{2i-l}{\mathbb E}\eta^l(n,Y_k(n))\binom{2i}{l}
&=& o({\mathbb E}Y_k^{2i-2}(n))
+\sum_{l=3}^{2i}{\mathbb E}Y_k^{2i-l}(n)o(n^{(l-2)/2})
\end{eqnarray*}
as $k$, $n\to\infty$, $k\le n$.
Substituting this into \eqref{incr.n.2i.x} and
taking into account \eqref{Y.to.infty}, we deduce that
\begin{eqnarray}\label{incr.n.2i}
{\mathbb E}\{Y^{2i}_{k+1}(n)-Y^{2i}_k(n)\}
&=& i[2\mu+(2i-1)b+o(1)]{\mathbb E}Y_k^{2i-2}(n)\nonumber\\
&&+\sum_{l=3}^{2i} {\mathbb E}Y_k^{2i-l}(n)o(n^{(l-2)/2}).
\end{eqnarray}
In particular, for $j=2i=4$ we get
\begin{eqnarray*}
{\mathbb E}\{Y^4_{k+1}(n)-Y^4_k(n)\}
&=& 2(2\mu+3b){\mathbb E}Y_k^2(n)
+{\mathbb E}Y_k(n) o(\sqrt n)+o(n)\\
&\sim& 2(2\mu+3b)(2\mu+b)n,
\end{eqnarray*}
due to \eqref{asy.2}. It implies that
\begin{eqnarray*}
{\mathbb E}Y^4_n(n) &\sim& (2\mu+3b)(2\mu+b)n^2
\quad\mbox{ as }n\to\infty.
\end{eqnarray*}
By induction arguments, we deduce from \eqref{incr.n.2i}
that, as $n\to\infty$,
\begin{eqnarray*}
{\mathbb E}Y^{2i}_n(n)
&\sim& n^i\prod_{k=1}^{i} (2\mu+(2k-1)b),
\end{eqnarray*}
which yields that $Y^2_n(n)/n$ weakly converges to Gamma
distribution with mean $2\mu+b$ and variance $2b(2\mu+b)$.
Together with \eqref{BK} this completes the proof.
\end{proof}

\mysection{Integral renewal theorem for transient chain}
\label{sec:renewal}

If the Markov chain $X_n$ is transient then
it visits any bounded set at most finitely many times.
The next result is devoted to the asymptotic behaviour
of the renewal functions
\begin{eqnarray*}
H_y(x) &:=& \sum_{n=0}^\infty \P_y\{X_n\le x\},\\
H(x) &:=&
\sum_{n=0}^\infty \P\{X_n\le x\}=\int H_y(x)\P\{X_0\in dy\}.
\end{eqnarray*}

\begin{lemma}\label{l:Hy.above}
Let the conditions \eqref{rec.4} and \eqref{T.above.cond.3} hold.
If
\begin{eqnarray}\label{2m1m2.1}
\sup_x(2xm_1(x)+m_2(x)) &<& \infty,\\
2xm_1(x)+m_2(x) \ge \varepsilon &>& 0\ \mbox{ ultimately in }x,
\label{2m1m2.2}
\end{eqnarray}
then there exists $c<\infty$ such that
$H_y(x)\le c(1+x^2)$ for all $y$ and $x$.
\end{lemma}

\begin{proof}
Fix $A>1$. After the stopping time
$T(Ax)=\min\{n\ge 1:X_n>Ax\}$ the chain falls
down below the level $x$ with probability not higher
than $1/A^\delta$, see Lemma \ref{est.for.return}
(where the condition \eqref{rec.1} follows from
\eqref{2m1m2.1} and \eqref{2m1m2.2}).
Hence, by the Markov property, for any $y$ we have
the following upper bound
\begin{eqnarray}\label{estimate.for.Hy}
H_y(x) &\le& \E_y\sum_{n=0}^{T(Ax)-1} \I\{X_n\le x\}
+ \frac{1}{A^\delta}\sup_{z\le x}H_z(x).
\end{eqnarray}
Therefore,
\begin{eqnarray*}
\sup_{y\ge 0}H_y(x) &\le& (1-1/A^\delta)\sup_y\E_y T(Ax)\\
&\le& (1-1/A^\delta)c_1(1+x^2)
\end{eqnarray*}
for some $c_1<\infty$, by Lemma \ref{lem.uniform}
(where the condition \eqref{T.above.cond.1} follows from
\eqref{2m1m2.1} and \eqref{2m1m2.2};
also $c(x)$ is bounded in \eqref{T.above.cond.2}).
The conclusion of the lemma is proven.
\end{proof}

\begin{theorem}\label{thm:renewal}
Let the conditions \eqref{rec.4}, \eqref{T.above.cond.1},
\eqref{T.above.cond.3}, and \eqref{mom.cond.ui} hold.
If $m_1(x)\sim\mu/x$ and $m_2(x)\to b>0$ as $x\to\infty$,
and $2\mu>b$, then, for any initial distribution of the chain $X$,
$$
H(x)\sim \frac{x^2}{2\mu-b}\ \mbox{ as }x\to\infty.
$$
\end{theorem}

\begin{proof}
Fix an arbitrary $y$. It follows from Lemma \ref{lem.uniform}
that $T(x)$ is finite a.s. for every $x$, so that the condition
\eqref{rec.2} holds and, by Theorem \ref{thm:transience},
$X_n\to\infty$ a.s. as $n\to\infty$. Then we may apply
Theorem \ref{thm:gamma} and state that $X_n^2/n$
weakly convergences to the $\Gamma$-distribution
with mean $2\mu+b$ and variance $(2\mu+b)2b$.
Thus, for every fixed $B$,
\begin{eqnarray*}
\sum_{n=0}^{[Bx^2]}\P_y\{X_n\le x\}
&=& \sum_{n=0}^{[Bx^2]}(\Gamma(x^2/n)+o(1))\\
&=& \sum_{n=0}^{[Bx^2]}\Gamma(x^2/n)+o(x^2).
\end{eqnarray*}
as $x\to\infty$. Since
\begin{eqnarray*}
\sum_{n=0}^{[Bx^2]}\Gamma(x^2/n) &\sim&
x^2\int_0^B\Gamma(1/z)dz\ \mbox{ as }x\to\infty
\end{eqnarray*}
and
\begin{eqnarray*}
\int_0^B\Gamma(1/z)dz &\to& \frac{1}{2\mu-b}\ \mbox{ as }B\to\infty,
\end{eqnarray*}
we conclude the lower bound
\begin{eqnarray}\label{bound.Hy.lower}
\liminf_{x\to\infty}\frac{H_y(x)}{x^2} &\ge& \frac{1}{2\mu-b}.
\end{eqnarray}

Let us now prove the upper bound
\begin{eqnarray}\label{bound.Hy.upper}
\limsup_{x\to\infty}\frac{H_y(x)}{x^2} &\le& \frac{1}{2\mu-b}.
\end{eqnarray}
Applying the upper bound of Lemma \ref{l:Hy.above}
on the right side of \eqref{estimate.for.Hy} we deduce that
\begin{eqnarray}\label{estimate.for.Hy.2}
H_y(x) &\le& \E_y\sum_{n=0}^{T(Ax)-1} \I\{X_n\le x\}
+ \frac{c}{A^\delta}(1+x^2).
\end{eqnarray}
For any $B$, the mean of the sum on the right of may be estimated
as follows:
\begin{eqnarray*}
\E_y\sum_{n=0}^{T(Ax)-1} \I\{X_n\le x\}
&\le& \E_y\Bigl\{\sum_{i=0}^{T(Ax)-1} \I\{X_n\le x\};
T(Ax)\le Bx^2\Bigr\}\\
&&\hspace{40mm}+\E_y\{T(Ax); T(Ax)>Bx^2\}.
\end{eqnarray*}
To estimate the second term we apply Lemma~\ref{lem.exp}
which yields
\begin{eqnarray*}
\E_y\{T(Ax); T(Ax)>Bx^2\}
&=& (Ax)^2\E_y\Bigl\{\frac{T(Ax)}{(Ax)^2};
\frac{T(Ax)}{(Ax)^2}>\frac{B}{A^2}\Bigr\}\\
&\le& (Ax)^2(B/A^2+1/c)e^{-c(B/A^2-t_0)}.
\end{eqnarray*}
Taking $B=2A^3$ we can ensure that
$$
\E_y\{T(Ax); T(Ax)>Bx^2\}\le c_1 e^{-cA} x^2.
$$
Hence,
\begin{eqnarray*}
H_y(x) &\le&
\E_y\Bigl\{\sum_{n=0}^{T(Ax)-1} \I\{X_n\le x\}; T(Ax)\le Bx^2\Bigr\}
+x^2O(A^{-\delta})\\
&\le& \sum_{n=0}^{[Bx^2]}\P_y\{X_n\le x\}+x^2O(A^{-\delta}).
\end{eqnarray*}
As already shown,
\begin{eqnarray*}
\sum_{n=0}^{[Bx^2]}\P_y\{X_n\le x\}
&=& x^2\int_0^B\Gamma(1/z)dz+o(x^2)\ \mbox{ as }x\to\infty,
\end{eqnarray*}
which implies the required upper bound \eqref{bound.Hy.upper}.
The lower \eqref{bound.Hy.lower} and upper \eqref{bound.Hy.upper}
bounds yield the equivalence, for every fixed $y$,
$$
H_y(x)\sim \frac{x^2}{2\mu-b}\ \mbox{ as }x\to\infty.
$$
Together with the uniform in $y$ estimate of Lemma
\ref{l:Hy.above} this completes the proof.
\end{proof}

\mysection{Construction of harmonic function}
\label{sec:harmony}

The Markov chain $X_n$ is assumed to be positive recurrent
with invariant measure $\pi$. Let $B$ be a Borel set in
$\Rp$ with $\pi(B)>0$; in our applications we consider
an interval $[0,x_0]$.
Denote $\tau_B:=\min\{n\ge 1:X_n\in B\}$. Since $X_n$ is
positive recurrent and $\pi(B)>0$,
$\E_x\tau_B<\infty$ for every $x$.

In this section we construct a {\it harmonic function} for $X_n$
killed at the time of the first visit to $B$, that is,
such a function $V(x)$ that, for every $x$,
$$
V(x)=\E_x\{V(X_1);X_1\notin B\}
\quad (=\E\{V(x+\xi(x));x+\xi(x)\notin B\}).
$$
If $V$ is harmonic then
\begin{eqnarray}\label{harm.n}
V(x)=\E_x\{V(X_n);\tau_B>n\}\ \mbox{ for every }n.
\end{eqnarray}
For any function $U(x):\Rp\to\R$,
denote its mean drift function by
$$
u(x):=\E_xU(X_1)-U(x)=\E U(x+\xi(x))-U(x).
$$

\begin{lemma}\label{l:V.harm}
Let $U\ge 0$, $U$ be zero on $B$, and
\begin{eqnarray}\label{Ex.for.sum.tau}
\E_x\sum_{n=0}^{\tau_B-1}(u(X_n))^+<\infty
\ \mbox{ for every }x.
\end{eqnarray}
Then the function
$$
V(x):=U(x)+\E_x\sum_{n=0}^{\tau_B-1}u(X_n)
$$
is well-defined, nonnegative and harmonic.
\end{lemma}

\begin{proof}
The condition \eqref{Ex.for.sum.tau} and the finiteness of
$\E_x\tau_B$ ensure that
\begin{eqnarray}\label{V.N}
\E_x\sum_{n=0}^{\tau_B-1}u(X_n) &=&
\lim_{N\to\infty}\E_x\sum_{n=0}^{(\tau_B-1)\wedge N}u(X_n).
\end{eqnarray}
Let ${\mathcal F}_n=\sigma\{X_0,\ldots,X_n\}$. We have
\begin{eqnarray*}
\E_x\sum_{n=0}^{(\tau_B-1)\wedge N}u(X_n)
&=& \E_x\sum_{n=0}^N u(X_n)\I\{\tau_B>n\}\\
&=& \E_x\sum_{n=0}^N \E\{U(X_{n+1})-U(X_n)\mid {\mathcal F}_n\}
\I\{\tau_B>n\}\\
&=& \E_x\sum_{n=0}^N \E\{(U(X_{n+1})-U(X_n))\I\{\tau_B>n\}\mid {\mathcal F}_n\},
\end{eqnarray*}
because $\I\{\tau_B>n\}\in{\mathcal F}_n$. By the fact that
$U$ is zero on $B$, we deduce that
$U(X_{n+1})\I\{\tau_B=n+1\}=0$ so that
\begin{eqnarray*}
\E_x\sum_{n=0}^{(\tau_B-1)\wedge N}u(X_n)
&=& \E_x\sum_{n=0}^N (U(X_{n+1})\I\{\tau_B>n+1\}-U(X_n)\I\{\tau_B>n\})\\
&=& \E_x U(X_{N+1})\I\{\tau_B>N+1\}-U(x),
\end{eqnarray*}
which together with \eqref{V.N} implies that
\begin{eqnarray}\label{V.nonneg}
U(x)+\E_x\sum_{n=0}^{\tau_B-1}u(X_n) &=&
\lim_{N\to\infty} \E_x U(X_{N+1})\I\{\tau_B>N+1\}.
\end{eqnarray}
The latter limit is nonnegative, since $U\ge 0$.
Together with the condition \eqref{Ex.for.sum.tau}
it implies that the mean of the left of \eqref{V.N}
is finite and the function $V$ is well-defined and,
as the representation \eqref{V.nonneg} shows, nonnegative.
(Also, nonnegativity follows from Theorem 14.2.2 from \cite{MT}
but we here produced self-contained short proof.)

Now prove that $V$ is harmonic. Since $U$ is zero on $B$,
$$
\E_x\{U(X_1);X_1\notin B\}=\E U(X_1)=U(x)+u(x).
$$
Therefore,
\begin{eqnarray*}
\E_x\{V(X_1);X_1\notin B\} &=& \E_x\{U(X_1);X_1\notin B\}
+\E_x\Bigl\{\E\Bigl\{\sum_{n=1}^{\tau_B-1}u(X_n)\Big|X_1\Bigr\};
X_1\notin B\Bigr\}\\
&=& U(x)+u(x)+\E_x\Bigl\{\E\Bigl\{\sum_{n=1}^{\tau_B-1}u(X_n)\I\{X_1\notin B\}
\Big| X_1\Bigr\}\Bigr\}\\
&=& U(x)+u(x)+\E_x\sum_{n=1}^{\tau_B-1}u(X_n)\I\{X_1\notin B\}\\
&=& U(x)+u(x)+\E_x\sum_{n=1}^{\tau_B-1}u(X_n)=V(x),
\end{eqnarray*}
so that $V$ is harmonic which completes the proof.
\end{proof}

\begin{lemma}\label{l:VUUtilde}
Suppose the functions $U_1$ and $U_2$ are both
locally bounded, equal to zero on $B$, positive on the
complement of $B$ and $U_1(x)\sim U_2(x)$ as $x\to\infty$.
If both satisfy the condition \eqref{Ex.for.sum.tau},
then $V_1(x)=V_2(x)$ for all $x$.
\end{lemma}

\begin{proof}
As stated in the previous proof, the condition \eqref{Ex.for.sum.tau}
and the finiteness of $\E_x\tau_B$ ensure that
\begin{eqnarray}\label{V.nonneg.1}
V_k(x) &=& \lim_{N\to\infty} \E_x\{U_k(X_{N+1});\tau_B>N+1\},\quad k=1,\ 2.
\end{eqnarray}
It suffices to prove that the limit in \eqref{V.nonneg.1}
is the same for $k=1$, $2$. Indeed, for every $A$,
\begin{eqnarray*}
\E_x\{U_k(X_{N+1});\tau_B>N+1\}
&=& \E_x\{U_k(X_{N+1});\tau_B>N+1,X_{N+1}\le A\}\\
&&\hspace{10mm} +\E_x\{U_k(X_{N+1});\tau_B>N+1,X_{N+1}>A\}.
\end{eqnarray*}
The first expectation on the right is not greater than
$$
\sup_{x\le A}U_k(x)\P_x\{\tau_B>N+1\}\to 0
\quad\mbox{ as }N\to\infty,
$$
because $U_k$ is locally bounded. As far as we consider
the second expectation, for every $\varepsilon>0$ the exists
sufficiently large $A$ such that
$$
(1-\varepsilon)U_1(x)\le U_2(x)\le(1+\varepsilon)U_1(x)
$$
and then
\begin{eqnarray*}
\lefteqn{(1-\varepsilon)\E_x\{U_1(X_{N+1});\tau_B>N+1,X_{N+1}>A\}}\\
&&\hspace{15mm} \le \E_x\{U_2(X_{N+1});\tau_B>N+1,X_{N+1}>A\}\\
&&\hspace{40mm} \le (1+\varepsilon)\E_x\{U_1(X_{N+1});\tau_B>N+1,X_{N+1}>A\}.
\end{eqnarray*}
These observations prove that the limits in \eqref{V.nonneg.1}
are equal for $k=1$, $2$ and the proof is complete.
\end{proof}

\mysection{Proof of Theorem \ref{thm:tail}}
\label{sec:proof}

Fix $x_0$ as in \eqref{def.B}. Consider the following function
$U$: $U=0$ on $[0,x_0]$ and
\begin{equation}\label{def.u}
U(x):=\int_{x_0}^x e^{R(y)}dy\ \mbox{ for } x\ge 0,
\ \mbox{ where }R(y)=\int_0^y r(z)dz.
\end{equation}
Note that the function $U$ solves the equation $U''-rU'=0$.
In other words, $U$ is harmonic function for a diffusion with drift
$r(x)$ and diffusion coefficient $1$ killed at leaving $(x_0,\infty)$.
According to our assumptions,
$$
r(z)=\frac{2\mu}{b}\frac{1}{z}+\frac{\varepsilon(z)}{z},
$$
where $\varepsilon(z)\to0$ as $z\to\infty$.
In view of the representation theorem, there exists a slowly
varying at infinity function $\ell(x)$ such that
$e^{R(x)}=x^{\rho-1}\ell(x)$
and $U(x)\sim x e^{R(x)}/\rho\sim x^\rho \ell(x)/\rho$ where $\rho=2\mu/b+1>2$.

For every $C\in\R$, define $U_C(x)=0$ on $[0,x_0]$ and
$$
U_C(x)=U(x)+Ce^{R(x)}\quad\mbox{ for }x>x_0.
$$

\begin{lemma}\label{L.harm2}
Assume the conditions of Theorem {\rm\ref{thm:tail}} hold. Then
$$
\E U_C(x+\xi(x))-U_C(x)
=\bigl((\rho-1)b(C_0-C)/2+o(1)\bigr)e^{R(x)}/x^2
\quad\mbox{ as }x\to\infty,
$$
where $C_0:=m_3(\rho-2)/3b$.
\end{lemma}
\begin{proof}
We start with the following decomposition:
\begin{eqnarray}\label{L.harm2.1}
\E U(x+\xi(x))-U(x)
&=& \E\{U(x+\xi(x))-U(x); |\xi(x)|\le \varepsilon x\}\nonumber\\
&&\hspace{2mm}+\E\{U(x+\xi(x))-U(x);\varepsilon x\le\xi(x)\le A x\}\nonumber\\
&&\hspace{7mm}+\E\{U(x+\xi(x))-U(x);\xi(x)>A x\}\nonumber\\
&&\hspace{10mm}+\E\{U(x+\xi(x))-U(x);\xi(x)<-\varepsilon x\}\nonumber\\
&=:& E_1+E_2+E_3+E_4.
\end{eqnarray}
The second and forth terms on the right may be bounded as follows:
\begin{eqnarray}\label{L.harm2.1.2a}
E_2+E_4 &\le& U((1+A)x)\P\{|\xi(x)|>\varepsilon x\}\nonumber\\
&\le& c_1U(x)\P\{|\xi(x)|>\varepsilon x\}\nonumber\\
&=& o(U(x)/x^3)\quad\mbox{ as }x\to\infty,
\end{eqnarray}
by the regular variation of $U$ and by the condition \eqref{moment_cond1}.
For the third term we have
\begin{eqnarray}\label{L.harm2.1.2}
E_3 &\le& \E\{U((1/A+1)\xi(x));\xi(x)>A x\}\nonumber\\
&\le& c_1\E\{\xi^{2\mu/b+1+\delta/2}(x);\xi(x)>A x\}\nonumber\\
&\le& c_1(Ax)^{-2-\delta/2}\E\{\xi^{2\mu/b+3+\delta}(x);\xi(x)>A x\}\nonumber\\
&=& o(U(x)/x^3)\quad\mbox{ as }x\to\infty,
\end{eqnarray}
due to the regular variation of $U$ and \eqref{cond.xi.U+2}.
To estimate the first term on the right side of \eqref{L.harm2.1},
we apply Taylor's formula:
\begin{eqnarray}\label{L.harm2.2}
E_1 &=& U'(x)\E\{\xi(x);|\xi(x)|\le \varepsilon x\}
+\frac{U''(x)}{2}\E\{\xi^2(x);|\xi(x)|\le \varepsilon x\}\nonumber\\
&&\hspace{40mm}+\frac{1}{6}\E\{U'''(x+\theta\xi(x))\xi^3(x);|\xi(x)|\le \varepsilon x\}.
\end{eqnarray}
where $0\le\theta=\theta(x,\xi(x))\le 1$.
By the construction of $U$ and the condition \eqref{r-cond.2},
\begin{eqnarray}\label{L.harm2.2.1}
U'(x)m_1(x)+\frac{U''(x)}{2}m_2(x)
&=& \frac{m_2(x)e^{R(x)}}{2}
\Bigl(\frac{2m_1(x)}{m_2(x)}+r(x)\Bigr)\nonumber\\
&=& O(e^{R(x)}/x^{2+\delta}).
\end{eqnarray}
Notice that
$$
\left|m_1(x)-\E\{\xi(x);|\xi(x)|\le \varepsilon x\}\right|
\le c_2\E|\xi(x)|^{3+\delta}/x^{2+\delta},
$$
and
$$
0\le m_2(x)-\E\{\xi^2(x);|\xi(x)|\le \varepsilon x\}
\le c_2\E|\xi(x)|^{3+\delta}/x^{1+\delta}.
$$
Applying now the condition \eqref{moment_cond1}, the relations
\eqref{L.harm2.2.1}, $U'(x)=e^{R(x)}$ and $U''(x)=O(e^{R(x)}/x)$,
we obtain
\begin{eqnarray}\label{L.harm2.2.2}
U'(x)\E\{\xi(x);|\xi(x)|\le \varepsilon x\}
+\frac{U''(x)}{2}\E\{\xi^2(x);|\xi(x)|\le \varepsilon x\}
&=& o(e^{R(x)}/x^2).
\end{eqnarray}

We next note that (\ref{3.moment}), our assumptions on $r(x)$
and the convergence
$$
\left|\E\{\xi^3(x);|\xi(x)|\le \varepsilon x\}-\E\xi^3(x)\right|
\to 0\quad\mbox{ as }x\to\infty,
$$
imply that
\begin{eqnarray}\label{3.deriv}
U'''(x)\E\{\xi^3(x);|\xi(x)|\le \varepsilon x\}
&=& (r^2(x)+r'(x))e^{R(x)}(\E\xi^3(x)+o(1))\nonumber\\
&=& ((\rho-1)(\rho-2)m_3+o(1))e^{R(x)}/x^2,
\end{eqnarray}
and
\begin{eqnarray}\label{3.deriv.2}
|\E\{(U'''(x+\theta\xi(x))-U'''(x))\xi^3(x);|\xi(x)|\le \varepsilon x\}|
&\le& c_3\varepsilon e^{R(x)}/x^2.
\end{eqnarray}

Substituting \eqref{L.harm2.2.2}, \eqref{3.deriv} and \eqref{3.deriv.2}
into \eqref{L.harm2.2} we get, for sufficiently large $x$,
\begin{eqnarray}\label{L.harm2.3}
\Bigl|E_1-\frac{(\rho-1)(\rho-2)}{6}m_3 e^{R(x)}/x^2\Bigr|
&\le& (c_3+1)\varepsilon e^{R(x)}/x^2.
\end{eqnarray}
It its turn, \eqref{L.harm2.1.2} and \eqref{L.harm2.3}
being implemented in \eqref{L.harm2.1} lead to
\begin{equation}\label{final1}
\E U(x+\xi(x))-U(x)
=\frac{(\rho-1)(\rho-2)m_3}{6}e^{R(x)}/x^2+o(e^{R(x)}/x^2),
\end{equation}
since $\varepsilon>0$ may be chosen as small as we please.

Applying similar arguments to the function $e^{R(x)}$, we get
\begin{equation}
\label{final2}
\E e^{R(x+\xi(x))}-e^{R(x)}
=-\frac{(\rho-1)b}{2}e^{R(x)}/x^2+o(e^{R(x)}/x^2).
\end{equation}
Combining \eqref{final1} and \eqref{final2} we arrive at
$$
\E U_C(x+\xi(x))-U_C(x)
=\frac{\rho-1}{2}((\rho-2)m_3/3-bC+o(1))
e^{R(x)}/x^2\quad\mbox{ as }x\to\infty,
$$
which completes the proof of the lemma.
\end{proof}

\begin{lemma}\label{L.harm.3}
Under the conditions of Theorem {\rm\ref{thm:tail}}, the harmonic
function $V$ generated by $U$ possesses the following decomposition:
$$
V(x)=U(x)+C_0e^{R(x)}+o(e^{R(x)})\quad\mbox{ as }x\to\infty.
$$
In particular, $V(x)>0$ ultimately in $x$.
\end{lemma}

\begin{proof}
Fix $\varepsilon>0$ and take $C:=C_0+\varepsilon$.
According to Lemma~\ref{L.harm2},
$$
u_C(x):=\E U_C(x+\xi(x))-U_C(x)
=(-(\rho-1)b\varepsilon/2+o(1))e^{R(x)}/x^2.
$$
Therefore, there exist $c_1<\infty$ and $x_1>x_0$ such that
$$
u_C(x)\le
\left\{
\begin{array}{ll}
c_1&\mbox{ if }x\le x_1,\\
0&\mbox{ if }x>x_1.
\end{array}
\right.
$$
Hence,
\begin{eqnarray*}
\E_x\sum_{n=0}^{\tau_B-1} u_C(X_n)
&\le& c_1\E_x\sum_{n=0}^{\tau_B-1}\I\{X_n\le x_1\}\\
&\le& c_1\sup_{x\le x_1}\E_x\tau_B=:c_2<\infty.
\end{eqnarray*}
Since $U_C(x)\sim U(x)$ as $x\to\infty$, by Lemma \ref{l:VUUtilde}
\begin{eqnarray*}
V(x) &=& U_C(x)+\E_x\sum_{n=0}^{\tau_B-1} u_C(X_n)\\
&\le& U_C(x)+c_2\\
&=& U(x)+(C_0+\varepsilon)e^{R(x)}+c_2.
\end{eqnarray*}
The arbitrary choice of $\varepsilon>0$ yields
\begin{eqnarray*}
V(x) &\le& U(x)+(C_0+o(1))e^{R(x)}\quad\mbox{ as }x\to\infty.
\end{eqnarray*}
Since $V\ge 0$,
\begin{equation}\label{finite.CC}
\E_x\sum_{n=0}^{\tau_B-1} e^{R(X_n)}/X_n^2<\infty
\end{equation}
for every $x$ because
$$
\E_x\sum_{n=0}^{\tau_B-1} u_C(X_n) \ge -U_C(x)>-\infty.
$$

Now take $C:=C_0-\varepsilon$. Again by Lemma~\ref{L.harm2},
$$
u_C(x):=\E U_C(x+\xi(x))-U_C(x)
=((\rho-1)b\varepsilon/2+o(1))e^{R(x)}/x^2,
$$
and the condition \eqref{Ex.for.sum.tau} holds due to
\eqref{finite.CC}. Then symmetric arguments lead to the lower bound
\begin{eqnarray*}
V(x) &\ge& U(x)+(C_0+o(1))e^{R(x)}\quad\mbox{ as }x\to\infty.
\end{eqnarray*}
Combining altogether we get the stated decomposition for $V(x)$.
\end{proof}

Having the harmonic function $V$ generated by $U$
we can define a new Markov chain $\widehat X_n$
on $\Rp$ with the following transition kernel
$$
\P_z\{\widehat X_1\in dy\}=\frac{V(y)}{V(z)}\P_z\{X_1\in dy;\tau_B>1\}
$$
if $V(z)>0$ and $\P_z\{\widehat X_1\in dy\}$ being arbitrary defined
if $V(z)=0$. Since $V$ is harmonic, then we also have
\begin{align}\label{connection}
\P_z\{\widehat X_n\in dy\}=\frac{V(y)}{V(z)}\P_z\{X_n\in dy;\tau_B>n\}
\ \mbox{ for all }n.
\end{align}

As well-known (see, e.g. \cite[Theorem 10.4.9]{MT}) the invariant measure
$\pi$ possesses the equality
\begin{equation}\label{pi-def}
\pi(dy)=\int_B\pi(dz)\sum_{n=0}^\infty \P_z\{X_n\in dy;\tau_B>n\}.
\end{equation}
Combining (\ref{connection}) and (\ref{pi-def}), we get
\begin{eqnarray*}
\pi(dy) &=& \frac{1}{V(y)}\int_B\pi(dz)V(z)
\sum_{n=0}^\infty \P_z\{\widehat X_n\in dy\}\\
&=& \frac{\widehat H(dy)}{V(y)}\int_B\pi(dz)V(z),
\end{eqnarray*}
where $\widehat H$ is the renewal measure generated by
the chain $\widehat X_n$ with initial distribution
$$
\P\{\widehat X_0\in dz\}=\widehat c\pi(dz)V(z),\ z\in B
\quad\text{and }\widehat c:=\Bigl(\int_B\pi(dz)V(z)\Bigr)^{-1}.
$$
Therefore,
\begin{eqnarray*}
\pi(x,\infty) &=& \widehat c
\int_x^\infty\frac{1}{V(y)}d\widehat H(y)\\
&\sim& \widehat c
\int_x^\infty\frac{1}{U(y)}d\widehat H(y)\ \mbox{ as }x\to\infty,
\end{eqnarray*}
since $V(x)\sim U(x)$ owing to Lemma \ref{L.harm2}.
After integration by parts we deduce
\begin{eqnarray}\label{repres}
\pi(x,\infty) &\sim& \widehat c\Bigl(
-\frac{\widehat H(x)}{U(x)}+\int_x^\infty\frac{\widehat H(y)U'(y)}{U^2(y)}dy\Bigr)\nonumber\\
&\sim& \widehat c\Bigl(
-\frac{\widehat H(x)}{U(x)}+\rho\int_x^\infty\frac{\widehat H(y)}{yU(y)}dy\Bigr)
\ \mbox{ as }x\to\infty.
\end{eqnarray}

In order to apply Theorem \ref{thm:renewal}
to the chain $\widehat X_n$, we have to show that its jumps
$\widehat\xi(x)$ satisfy the corresponding conditions.
By the construction, the absolute moments of order $2+\delta/2$
of $\widehat\xi(x)$ are uniformly bounded, because
\begin{eqnarray*}
\E|\widehat\xi(x)|^{2+\delta/2}
&=& \frac{1}{V(x)}\E|\xi(x)|^{2+\delta/2}V(x+\xi(x))\\
&=& \frac{1}{V(x)}(\E\{|\xi(x)|^{2+\delta/2}V(x+\xi(x));\xi(x)\le Ax\}\\
& &\hspace{1cm}+\E\{|\xi(x)|^{2+\delta/2}V(x+\xi(x));\xi(x)>Ax\})\\
&\le& \frac{V((1+A)x)}{V(x)}\E|\xi(x)|^{2+\delta/2}\\
&&\hspace{20mm}+\frac{1}{V(x)}
\E\{|\xi(x)|^{2+\delta/2}V((1+1/A)\xi(x));\xi(x)>Ax\},
\end{eqnarray*}
where $A$ is from the condition \eqref{cond.xi.U+2}.
Here the first term on the right side is bounded due to the condition
\eqref{moment_cond1} and regular variation of $V$ with index
$\rho$ and the second one is bounded by \eqref{cond.xi.U+2},
because
\begin{eqnarray*}
\E\{|\xi(x)|^{2+\delta/2}V((1+1/A)\xi(x));\xi(x)>Ax\}
&\le& \frac{c_4}{x^{\delta/4}}\E\{|\xi(x)|^{2+\delta+\rho};\xi(x)>Ax\}\\
&\le& c_5 x^{\rho-1-\delta/4}=o(V(x)/x).
\end{eqnarray*}
Then, in particular, the condition \eqref{mom.cond.ui} of
existence of integrable majorant for the
squares of jumps $\widehat\xi(x)$ and the condition
\eqref{rec.4} follow. Also it implies that
\begin{equation}\label{exp6}
\lim_{x\to\infty}\E\widehat\xi^2(x)
=\lim_{x\to\infty}\E\xi^2(x)=b.
\end{equation}

Further, the boundedness of the moments of order $2+\delta/2$
of $\widehat\xi(x)$ yields that, for every $\varepsilon>0$,
\begin{eqnarray}\label{exp1}
\E\widehat\xi(x) &=&
\E\{\widehat\xi(x);|\widehat\xi(x)|\le\varepsilon x\}+o(1/x)\nonumber\\
&=&\frac{1}{V(x)}\E\{\xi(x)V(x+\xi(x));|\xi(x)|\le\varepsilon x\}
+o(1/x).
\end{eqnarray}
Fix $\varepsilon_1>0$. Recalling that, by Lemma \ref{L.harm.3},
the function $V(x)-U(x)\sim C_0e^{R(x)}$ is regularly varying
with index $\rho-1$, we may choose $\varepsilon>0$ so small that
\begin{eqnarray}\label{difference.VU}
|V(x+y)-U(x+y)-(V(x)-U(x))| &\le& \varepsilon_1 e^{R(x)}
\quad\mbox{ for all }|y|\le\varepsilon x.
\end{eqnarray}
Then
$$
\E\{\xi(x)(V(x+\xi(x))-V(x));|\xi(x)|\le\varepsilon x\}
$$
differs from
$$
\E\{\xi(x)(U(x+\xi(x))-U(x));|\xi(x)|\le\varepsilon x\}
$$
by the quantity not greater than $\varepsilon_1 e^{R(x)}\E|\xi(x)|$.
Using Taylor's formula and the relation
$$
\sup_{|y|\le x/2}U''(x+y)
=\sup_{|y|\le x/2}r(x+y)e^{R(x+y)}=O(e^{R(x)}/x),
$$
we get
\begin{eqnarray*}
& &\E\bigl\{\xi(x)(U(x+\xi(x))-U(x));|\xi(x)|\le\varepsilon x\bigr\}\\
& &\hspace{1cm}= U'(x)\E\{\xi^2(x);|\xi(x)|\le\varepsilon x\}
+O(e^{R(x)}/x).
\end{eqnarray*}
It follows now from the condition \eqref{moment_cond1}
that the asymptotics of truncated expectations of the first
and the second order coincide with that of full expectations.
Combining altogether and relations $V(x)\sim U(x)$ and
$U'(x)=e^{R(x)}\sim \rho U(x)/x$, we deduce that
$$
\limsup_{x\to\infty}\Bigl|
\frac{x}{V(x)}\E\{\xi(x)V(x+\xi(x));|\xi(x)|\le\varepsilon x\}
-(-\mu+\rho b)\Bigr| \le \varepsilon_1\rho\sup_x\E|\xi(x)|.
$$
Plugging this into (\ref{exp1}) and recalling that
$\rho=1+2\mu/b$, we conclude that
$$
\limsup_{x\to\infty}|x\E\widehat\xi(x)-(\mu+b)|
\le \varepsilon_1\rho\sup_x\E|\xi(x)|.
$$
Since $\varepsilon_1>0$ may be chosen as small as we please,
\begin{equation}\label{exp5}
x\E\widehat\xi(x) \to \mu+b\quad\mbox{ as }x\to\infty.
\end{equation}

Finally, check the condition \eqref{T.above.cond.3}
for the chain $\widehat X_n$. As already shown,
$$
2x\widehat m_1(x)+\widehat m_2(x)\to 2(\mu+b)+b=2\mu+3b>0
\ \mbox{ as }x\to\infty
$$
It allows us to choose $x_1>x_0$ so that $U(x_1)>0$,
$V(x)\ge U(x)/2$ for all $x>x_1$ (this is possible because
$V(x)\sim U(x)$) and
$$
\inf_{x>x_1}(2x\widehat m_1(x)+\widehat m_2(x))>0.
$$
Then the condition \eqref{T.above.cond.3} holds with $x_1$
instead of $x_0$. Indeed, by the construction,
$\widehat X_n>x_0$ for any $n\ge 1$ which implies
\begin{eqnarray*}
\widehat H_y(x_0)=\sum_{n=0}^\infty \P_y\{\widehat X_n\le x_0\} &\le& 1.
\end{eqnarray*}
Further, as follows from \eqref{harm.n} and increase of the
function $U$, for every $x>x_0$,
\begin{eqnarray*}
V(x)=\E_x\{V(X_n);\tau_B>n\}
&\ge& \frac{U(x)}{2}\P_x\{X_n>x_1,\tau_B>n\}\\
&\ge& \frac{U(x_1)}{2}\P_x\{X_n>x_1,\tau_B>n\}.
\end{eqnarray*}
The role of the condition \eqref{exersions} is just to be
applied here; it guarantees that
$$
\inf_{x>x_0}V(x)>0.
$$
Therefore, for every $y\in[x_0,x_1]$,
\begin{eqnarray*}
\widehat H_y(x_1)=\sum_{n=0}^\infty\P_y\{\widehat X_n\le x_1\}
&=& \frac{1}{V(y)}\sum_{n=0}^\infty
\int_{x_0}^{x_1}V(z)\P_y\{X_n\in dz,\tau_B>n\}\\
&\le& \frac{\sup_{x_0<z\le x_1}V(z)}{\inf_{y>x_0}V(y)}
\sum_{n=0}^\infty \P_y\{\tau_B>n\}\\
&=& c\E_y\tau_B,
\end{eqnarray*}
and the latter mean value is bounded in $y\in[x_0,x_1]$.

Now it is shown that $\widehat X_n$ satisfies all the conditions
of Theorem \ref{thm:renewal},
so that $\widehat X_n$ is transient and
$$
\widehat H(x)\sim\frac{x^2}{2(\mu+b)-b}=
\frac{x^2}{2\mu+b}\ \mbox{ as }x\to\infty.
$$
Substituting this equivalence into \eqref{repres}
where $U(x)$ is regularly varying with index $\rho$
we arrive at the following equivalence:
\begin{eqnarray*}
\pi(x,\infty) &\sim& \frac{2}{(2\mu+b)(\rho-2)}\frac{x^2}{U(x)}
\int_B\pi(dz)V(z)\\
&\sim& \frac{2\rho}{(2\mu+b)(\rho-2)}xe^{-R(x)}\int_B\pi(dz)V(z)
\ \mbox{ as }x\to\infty.
\end{eqnarray*}
The proof of Theorem \ref{thm:tail} is complete.

\section{Proof of Theorem \ref{thm:tail2}}

In present section we work with the same function $U$ as defined
in the previous section. Now we should again prove that the
corresponding harmonic function $V$ is ultimately positive and
that $V(x)\sim U(x)$ as $x\to\infty$. Since here we do not assume
convergence of the third moments of jumps, we need to modify
our approach for proving these properties.

As in the previous section, for every $C\in\R$,
define $U_C(x)=0$ on $[0,x_0]$ and
$$
U_C(x)=U(x)+Ce^{R(x)}\quad\mbox{ for }x>x_0.
$$

\begin{lemma}\label{L.harm2.2.left}
Assume the conditions of Theorem {\rm\ref{thm:tail2}} hold.
Then there exist constants $C_1$, $C_2\in\R$ such that,
for all sufficiently large $x$,
\begin{eqnarray*}
\E U_{C_1}(x+\xi(x))-U_{C_1}(x) &<& 0,\\
\E U_{C_2}(x+\xi(x))-U_{C_2}(x) &>& 0.
\end{eqnarray*}
\end{lemma}

\begin{proof}
As the calculations in Lemma \ref{L.harm2} show,
without the condition on the convergence of the third moments
of jumps we still have the relation
\begin{eqnarray*}
\E U(x+\xi(x))-U(x)
&=& o(e^{R(x)}/x^2),
\end{eqnarray*}
which together with \eqref{final2} concludes the proof.
\end{proof}

The only place where the condition that the chain if left
skip-free is utilised is the following result.

\begin{lemma}\label{L.harm2.3.left}
Under the conditions of Theorem {\rm\ref{thm:tail2}},
the increments of the harmonic function $V$ generated by $U$
satisfy the following bounds: for $y>0$,
\begin{eqnarray*}
U(x+y)-U(x)+C_2(e^{R(x+y)}-e^{R(x)})
&\le& V(x+y)-V(x)\\
&\le& U(x+y)-U(x)+C_1(e^{R(x+y)}-e^{R(x)})
\end{eqnarray*}
ultimately in $x$. In particular, $V(x)\sim U(x)$ as $x\to\infty$
and $V(x)>0$ ultimately in $x$.
\end{lemma}

\begin{proof}
Both functions $U_{C_1}$ and $U_{C_2}$ satisfy the conditions of
Lemma \ref{l:VUUtilde} by the same arguments as in Lemma \ref{L.harm.3}.

Let $y>0$. Given $X_0=x+y$, denote $\tau_x:=\min\{n\ge1:X_n=x\}$.
Since the chain is left skip-free, $\tau_x<\tau_B$.
Having in mind that $u_{C_1}(X_n)<0$ before this stopping time,
we get, by the Markov property,
\begin{eqnarray*}
V(x+y)-V(x) &=& U_{C_1}(x+y)-U_{C_1}(x)
+\E_{x+y}\sum_{n=0}^{\tau_B} u_{C_1}(X_n)
-\E_x\sum_{n=0}^{\tau_B} u_{C_1}(X_n)\\
&\le& U_{C_1}(x+y)-U_{C_1}(x),
\end{eqnarray*}
and similarly $V(x+y)-V(x) \ge U_{C_2}(x+y)-U_{C_2}(x)$,
which completes the proof.
\end{proof}

We are now able to compute the mean drift of the transformed chain $\widehat{X}_n$.
We may just repeate the arguments from the proof of Theorem \ref{thm:tail}
with the inequality
\begin{eqnarray*}
|V(x+y)-U(x+y)-(V(x)-U(x))| &\le&
\max\{|C_1|,|C_2|\}(e^{R(x+y)}-e^{R(x)})
\end{eqnarray*}
instead of \eqref{difference.VU}. As a result we see that \eqref{exp5} is valid
under the conditions of Theorem~\ref{thm:tail2}.

All other parts of the derivation of the asymptotics of $\pi(x,\infty)$ can be taken
from the proof of Theorem~\ref{thm:tail} without any change.

\end{document}